\documentclass[12pt]{arxivclassfile}

\usepackage{mathrsfs,color}

\def\R{{\mathbb R}}

\def\d{{\rm d}}

\newcommand{\ee}{\mathrm{e}}
\newcommand{\ii}{\mathrm{i}}

\def\BB{B}

\def\ff{f}

\def\email#1{\ead{#1}}

\newcommand{\Z}{\mathbb{Z}}
\newcommand{\vv}{\boldsymbol{v}}

\newcommand{\Sch}{\mathscr{S}}

\newcommand{\loc}{\mathord{\dot{\bigtriangleup}}}
\newcommand{\rd}{\mathrm{d}}
\newcommand{\Div}{\nabla \cdot}
\newcommand{\Grad}{\nabla}

\newcommand{\supp}{{\rm supp\,}}
\newcommand{\abs}[1]{\left| #1 \right|}
\newcommand{\norm}[1]{\| #1 \|}
\newcommand{\bignorm}[1]{\left\| #1 \right\|}
\newcommand{\inner}[2]{( #1 , #2 )}
\newcommand{\dInd}{{\mathbf 1}}

\usepackage{graphicx,epsfig,psfrag}

\usepackage{mathscinet}
\usepackage{url}
\usepackage{amssymb}
\usepackage{amsmath}
\usepackage{amsthm}
\usepackage{enumerate}
\newtheorem{theorem}{Theorem}[section]

\newtheorem{corollary}[theorem]{Corollary}
\newtheorem{lemma}[theorem]{Lemma}
\newtheorem{proposition}[theorem]{Proposition}

\numberwithin{equation}{section}

\begin{document}

\begin{frontmatter}

%
\author[warwick]{David S. McCormick}
\ead{d.s.mccormick@warwick.ac.uk}
\author[EJO]{Eric J.\ Olson\corref{thing}}
\ead{eric@reno}
\author[warwick]{James C.\ Robinson}
\ead{J.C.Robinson@warwick.ac.uk}
\author[warwick]{Jose L.\ Rodrigo}
\ead{J.Rodrigo@warwick.ac.uk}
\address[warwick]{Mathematics Institute, Zeeman Building,\\ University of Warwick, Coventry CV4 7AL, UK.}
\address[EJO]{Department of Mathematics/084, University of Nevada, Reno, NV 89557. USA.}
\ead{ejolson@unr.edu}
\author[XJTLU]{Alejandro Vidal-L\'opez}
\ead{Alejandro.Vidal@xjtlu.edu.cn}
\address[XJTLU]{Department of Mathematical Sciences, Xi'an Jiaotong-Liverpool University, Suzhou 215123. China P.R.}
\author[Z]{Yi Zhou}
\email{yizhou@fudan.edu.cn}
\address[Z]{School of Mathematical Sciences, Fudan University, Shanghai 200433. China P.R.}


\title{Lower bounds on blowing-up solutions of the 3D Navier--Stokes equations in $\dot H^{3/2}$, $\dot H^{5/2}$, and $\dot B^{5/2}_{2,1}$.}

\begin{abstract}
If $u$ is a smooth solution of the Navier--Stokes equations on $\R^3$ with first blowup time $T$, we prove lower bounds for $u$ in the Sobolev spaces $\dot H^{3/2}$, $\dot H^{5/2}$, and the Besov space $\dot B^{5/2}_{2,1}$, with optimal rates of blowup: we prove the strong lower bounds $\|u(t)\|_{\dot H^{3/2}}\ge c(T-t)^{-1/2}$ and $\|u(t)\|_{\dot B^{5/2}_{2,1}}\ge c(T-t)^{-1}$, but in $\dot H^{5/2}$ we only obtain the weaker result $\limsup_{t\to T^-}(T-t)\|u(t)\|_{\dot H^{5/2}}\ge c$. The proofs involve new inequalities for the nonlinear term in Sobolev and Besov spaces, both of which are obtained using a dyadic decomposition of $u$.
\end{abstract}

%
%
%
%
%
%
%




\end{frontmatter}

\def\ri{{\rm i}}

\section{Introduction}

The aim of this paper is to prove lower bounds on smooth solutions of the three-dimensional Navier--Stokes equations, under the assumption that there is a finite `first blowup time' $T$. Results of this type date back to Leray (1934), who showed that there exists an absolute constant $c_1$ such that
$$
\|u(t)\|_{H^1}\ge \frac{c_1}{\sqrt{T-t}}.
$$
In fact this result, and all subsequent lower bounds, are a consequence of upper bounds on the local existence time for solutions with initial data in $H^1$. Leray also stated (without proof) the lower bound
$$
\|u(t)\|_{L^p}\ge \frac{c}{(T-t)^{(p-3)/2p}},
$$
a proof of which can be found in Giga (1986) and Robinson \& Sadowski (2012).

More recently there have been a number of papers that treat the problem of blowup in Sobolev spaces $\dot H^s$ for $s>1/2$. Benameur (2010, with a similar periodic analysis in 2013) showed that for $s>5/2$
$$
\|u(t)\|_{\dot H^s}\ge c_s\|u(T-t)\|_{L^2}^{(3-2s)/3}(T-t)^{-s/3},
$$
which was improved by Robinson, Sadowski, \& Silva (2012) to
\begin{equation}
\|u(t)\|_{\dot H^s}\ge\begin{cases}
c(T-t)^{-(2s-1)/4}&s\in(1/2,5/2),\ s\neq 3/2,\\
c\|u_0\|_{L^2}^{(5-2s)/5}(T-t)^{-s/5}&s>5/2.
\end{cases}\label{RSS}
\end{equation}
As argued by Robinson et al.\ (2012), the bound
$$
\|u(t)\|_{\dot H^s}\ge c(T-t)^{-(2s-1)/4}
$$
is what one would expect from scaling considerations for all $s>1/2$; we refer to this here as the `optimal rate'.

We note that in the bounds in (\ref{RSS}) the cases $s=3/2$ and $s=5/2$ are excluded, and that the bounds for $s>5/2$ are not at the optimal rate. Although Benameur and Robinson et al.\ both obtained the lower bound
$$
\|\hat u(t)\|_{L^1}\ge c(T-t)^{-1/2},
$$
i.e.\ a bound with the `optimal rate' in a space with the same scaling as $\dot H^{3/2}$, no lower bound with the correct rate in any space scaling like $\dot H^{5/2}$ has previously been shown.

Recently, Cortissoz, Montero, \& Pinilla (2014) proved lower bounds in $\dot H^{3/2}$ and $\dot H^{5/2}$ at the optimal rates but with logarithmic corrections,
$$
\|u(t)\|_{\dot H^{3/2}}\ge \frac{c}{\sqrt{(T-t)|\log(T-t)|}}\quad\mbox{and}\quad\|u(t)\|_{\dot H^{5/2}}\ge \frac{c}{(T-t)|\log(T-t)|},
$$
where in both cases $c$ depends on $\|u_0\|_{L^2}$.

In this paper we fill some of these gaps. We will show that if $u$ is a smooth solution with maximal existence time $T$ then
\begin{equation}\label{CZ}
\|u(t)\|_{\dot H^{3/2}}\ge\frac{c}{(T-t)^{1/2}},
\end{equation}
which we refer to as a `strong blowup estimate', and
$$
\limsup_{t\uparrow T^*}\, (T-t)\|u(t)\|_{\dot H^{5/2}}\ge c,
$$
which we refer to as a `weak blowup estimate'. We also prove a strong blowup estimate in the Besov space $\dot B^{5/2}_{2,1}$, which has the same scaling as $\dot H^{5/2}$, $$
\|u(t)\|_{\dot B^{5/2}_{2,1}}\ge \frac{c}{T-t}.
$$

The key to these bounds are two inequalities for the nonlinear term $B(u,u)=(u\cdot\nabla)u$. Both are proved using a dyadic decomposition of $u$. The first is the Sobolev space inequality
\[
|(\Lambda^s B(u,u),\Lambda^s u)| \leq c \|u\|_{\dot{H}^s}\|u\|_{\dot H^{s+1}} \|u\|_{\dot{H}^{3/2}},\qquad s\ge 1,
\]
valid whenever the right-hand side is finite (in fact we prove a more general commutator-type estimate in Proposition \ref{prop:NSInequalityNorm}). The second is the Besov bound
$$
|(\dot\triangle_kB(u,u),\dot\triangle_ku)|\le cd_k2^{-k(d/2+1)}\|u\|_{\dot B^{5/2}_{2,1}}^2\|\dot\triangle_ku\|_{L^2},
$$
where $c$ does not depend on $k$ and $\sum_kd_k=1$. We present the proofs of these inequalities in Sections \ref{sec:H} and \ref{sec:B}, with the resulting blowup estimates given first in Sections \ref{sec:HB} and \ref{sec:BB}.

 Within the ten days prior to the submission of this paper to the arXiv, two other papers were submitted providing proofs of the lower bound in (\ref{CZ}) for $\dot H^{3/2}$ - one by Cheskidov \& Zaya (using an alternative dyadic argument) and one by Montero (using a very neat interpolation argument).

\section{Preliminaries}

In this section we prove a simple ODE lemma that provides lower bounds on solutions that blow up, and we recall the dyadic decomposition that we will use to prove our Sobolev and Besov space inequalities.

\subsection{Lower bounds and differential inequalities}

Lower bounds on solutions that blowup at some time $T>0$ can be derived from differential inequalities for the norms of the solution (i.e.\ from upper bounds on the local existence time). The following simple ODE lemma makes this precise.

\begin{lemma}\label{easy}
  If $\dot X\le cX^{1+\gamma}$ and $X(t)\to\infty$ as $t\to T$ then
  \begin{equation}\label{lowerB}
  X(t)\ge\left(\frac{1}{\gamma c(T-t)}\right)^{1/\gamma}.
  \end{equation}
\end{lemma}

\begin{proof}
  Write the differential inequality as
  $$
\frac{\d X}{X^{1+\gamma}}\le c\,\d t
$$
and integrate from $t$ to $s$ to yield
$$
\frac{1}{X(t)^\gamma}-\frac{1}{X(s)^\gamma}\le\gamma c(s-t).
$$
Letting $s\to T$ yields (\ref{lowerB}).\end{proof}

\subsection{Homogeneous Sobolev spaces}

We denote by $\dot H^s(\R^n)$ the space
$$
\left\{u:\ \hat u\in L^1_{\rm loc}(\R^n):\ \int_{\R^n}|\xi|^{2s}|\hat u(\xi)|^2\,\d\xi<\infty\right\},
$$
where
\begin{equation}\label{FT}
\mathscr{F}[u](\xi) = \hat{u}(\xi) = \int_{\R^{n}} \ee^{-2\pi \ii x \cdot \xi} u(x) \, \rd x
\end{equation}
is the Fourier transform of $u$. We denote by $\Lambda^s$ the operator with Fourier multiplier $|\xi|^s$; then the norm in $\dot H^s$ is given by
$$
\|u\|_{\dot H^s}=\|\Lambda^su\|_{L^2}=\||\xi|^s\hat u(\xi)\|_{L^2}=\int_{\R^n}|\xi|^{2s}|\hat u(\xi)|^2\,\d k.
$$

\subsection{Homogenous Besov spaces}

Here we recall some of the standard theory of homogeneous Besov spaces which we will use throughout the paper; we refer the reader to Bahouri et al.\ (2011),  for proofs and many more details that we must omit.


For the purposes of this section, given a function $\phi$ and $j \in \Z$ we denote by $\phi_{j}$ the dilation
\[
\phi_{j}(\xi) = \phi(2^{-j}\xi).
\]
Let $\mathcal{C}$ be the annulus $\{ \xi \in \R^{n} : 3/4 \leq |\xi| \leq 8/3 \}$. There exist radial functions $\chi \in C^{\infty}_{c}(B(0, 4/3))$ and $\varphi \in C^{\infty}_{c}(\mathcal{C})$ both taking values in $[0,1]$ such that
\begin{subequations}
\label{eqn:DyadicPartition}
\begin{align}
&\text{for all } \xi \in \R^{n}, & \chi(\xi) + \sum_{j \geq 0} \varphi_{j}(\xi) &= 1, \label{eqn:DyadicPartition1} \\
&\text{for all } \xi \in \R^{n} \setminus \{ 0 \} , & \sum_{j \in \Z} \varphi_{j}(\xi) &= 1, \label{eqn:DyadicPartition2} \\
&\text{if } |j - j'| \geq 2, \text{ then } & \hspace{-0.5in} \supp \varphi_{j} \cap \supp \varphi_{j'} &= \varnothing, \label{eqn:DyadicPartition3} \\
&\text{if } j \geq 1, \text{ then } & \supp \chi \cap \supp \varphi_{j} &= \varnothing. \label{eqn:DyadicPartition4}
\end{align}
\end{subequations}
We let $h = \mathscr{F}^{-1} \varphi$ and $\widetilde{h} = \mathscr{F}^{-1} \chi$, where ${\mathscr F}^{-1}$ is the inverse of the Fourier transform operator defined in (\ref{FT}). 

Given a measurable function $\sigma$ defined on $\R^{n}$ with at most polynomial growth at infinity, we define the Fourier multiplier operator $M_{\sigma}$ by $M_{\sigma} u := \mathcal{F}^{-1} (\sigma \hat{u})$. For $j \in \Z$, the \emph{homogeneous dyadic blocks} $\loc_{j}$ and the homogeneous cut-off operator $\dot{S}_{j}$ are defined by setting
\begin{align*}
\loc_{j} u &= M_{\varphi_{j}} u = 2^{jn} \int_{\R^{n}} h(2^{j} y) u(x-y) \, \rd y\qquad\mbox{and} \\
\dot{S}_{j} u &= M_{\chi_{j}} u = 2^{jn} \int_{\R^{n}} \widetilde{h}(2^{j} y) u(x-y) \, \rd y.
\end{align*}
Formally, we can write the following \emph{Littlewood--Paley decomposition}
\[
\mathop{\mathrm{Id}} = \sum_{j \in \Z} \loc_{j}.
\]

We denote by $\Sch'_{h}(\R^{n})$ the space of tempered distributions such that
\[
\lim_{\lambda \to \infty} \norm{M_{\theta(\lambda \, \cdot \,)} u}_{L^{\infty}} = 0 \quad \text{ for any } \theta \in C^{\infty}_{c}(\R^{n}).
\]
Then the homogeneous decomposition makes sense in $\Sch'_{h}(\R^{n})$: if $u \in \Sch'_{h}(\R^{n})$, then $u = \lim_{j \to \infty} \dot{S}_{j} u$ in $\Sch'_{h}(\R^{n})$. Moreover, using the homogeneous decomposition, it is straightforward to show that
\[
\dot{S}_{j}u = \sum_{j' \leq j-1} \loc_{j'} u.
\]

Given a real number $s$ and two numbers $p, r \in [1, \infty]$, the \emph{homogeneous Besov space} $\dot{B}^{s}_{p,r}(\R^{n})$ consists of those distributions $u$ in $\Sch'_{h}(\R^{n})$ such that
\[
\norm{u}_{\dot{B}^{s}_{p,r}} := \bigg( \sum_{j \in \Z} 2^{rjs} \norm{\loc_{j} u}_{L^{p}}^{r} \bigg)^{1/r} < \infty
\]
if $r < \infty$, and
\[
\norm{u}_{\dot{B}^{s}_{p,\infty}} := \sup_{j \in \Z} 2^{js} \norm{\loc_{j} u}_{L^{p}} < \infty
\]
if $r = \infty$. For each of these spaces all choices of the function $\varphi$ used to define the blocks $\dot\triangle_j$ lead to equivalent norms and hence to the same space.

Note that if $u \in \Sch'_{h}(\R^{n})$ belongs to $\dot{B}^{s}_{p,r}(\R^{n})$ then
there exists a non-negative sequence $(d_{j})_{j \in \Z}$ such that
\begin{equation}
\label{eqn:StandardTrick}
\norm{\loc_{j} u}_{L^{p}} \leq d_{j} 2^{-js}\|u\|_{\dot B^s_{p,r}}\ \forall\ j\in\Z, \qquad \text{where} \qquad \norm{(d_{j})}_{\ell^{r}} = 1.
\end{equation}

\section{Blowup estimates in $\dot H^{3/2}$ (strong) and $\dot H^{5/2}$ (weak)}\label{sec:HB}

The proofs of the blow up results follows easily from upper bounds on the nonlinear term. We postpone a detailed presentation of the estimates and proofs of these bounds until Section \ref{sec:H}.  In this section we assume those estimates, and present a  straightforward proof of the strong blowup estimate in $\dot H^{3/2}$, and, with an additional contradiction argument,  of the weak blowup estimate in $\dot H^{5/2}$.

\begin{theorem}\label{thm:H32}
  Suppose that $u$ is a classical solution of the Navier--Stokes existence with maximal existence time $T$. Then
  \begin{equation}\label{H32lower}
\|u(T-t)\|_{\dot H^{3/2}}^2\ge c_{3/2}^{-2}t^{-1}.
\end{equation}
\end{theorem}

\begin{proof}
  We take the inner product of the equation with $u$ in $\dot H^{3/2}$, i.e.\ we apply $\Lambda^{3/2}$ and take the inner product with $\Lambda^{3/2}u$,
\begin{align*}
\frac{1}{2}\frac{\d}{\d t}\|u\|_{\dot H^{3/2}}^2+\|u\|_{\dot H^{5/2}}^2&=(\Lambda^{3/2}B(u,u),\Lambda^{3/2}u)\\
&\le c_{3/2}\|u\|_{\dot H^{3/2}}^2\|u\|_{\dot H^{5/2}},
\end{align*}
using the inequality
$$\abs{\inner{\Lambda^s[(u \cdot \nabla)u]}{\Lambda^su}} \leq c \|u\|_{\dot{H}^s}\|u\|_{\dot H^{s+1}} \|u\|_{\dot{H}^{3/2}}\qquad s\ge 1,
$$
from (\ref{whatwewant})  with $s=3/2$, which is proved in Section \ref{sec:H}. We use Young's inequality on the right-hand side to obtain
$$
\frac{\d}{\d t}\|u\|_{\dot H^{3/2}}^2+\|u\|_{\dot H^{5/2}}^2\le c_{3/2}^2\|u\|_{\dot H^{3/2}}^4.
$$
Dropping the second term on the left-hand side, the required lower bound follows immediately from Lemma \ref{easy}.\end{proof}

We now use a contradiction argument to obtain a weak lower bound in $\dot H^{5/2}$ at the correct rate.

\begin{theorem}
  Suppose that $u$ is a classical solution of the Navier--Stokes existence with maximal existence time $T$. Then
  \begin{equation}\label{limsup}
  \limsup_{t\uparrow T}\,(T-t)\|u(t)\|_{\dot H^{5/2}}\ge c.
  \end{equation}
\end{theorem}

\begin{proof}
  We proceed by contradiction, and suppose that for $\tau\le t\le T$,
\begin{equation}\label{H32lower}
\|u(t)\|_{\dot H^{5/2}}\le\varepsilon(T-t)^{-1},
\end{equation}
where $\varepsilon$ is chosen so that $2c_{3/2}\varepsilon<1$. Then on this interval
$$
\frac{1}{2}\frac{\d}{\d t}\|u\|_{\dot H^{3/2}}^2\le  c_{3/2}\|u\|_{\dot H^{3/2}}^2\|u\|_{\dot H^{5/2}}-\|u\|_{\dot H^{5/2}}^2.
$$
Since $ax-x^2$ is increasing in $x$ while $x\le a/2$, and by assumption
$$
\|u(t)\|_{\dot H^{5/2}}\le\frac{\varepsilon}{T-t}\le \frac{\frac{1}{2}c_{3/2}^{-1}}{T-t}\le \frac{1}{2}\left[c_{3/2}\|u(t)\|_{\dot H^{3/2}}^2\right],
$$
it follows that
$$
\frac{\d}{\d t}\|u\|_{\dot H^{3/2}}^2\le 2c_{3/2}\|u\|_{\dot H^{3/2}}^2\frac{\varepsilon}{T-t}-\frac{2\varepsilon^2}{(T-t)^{-2}}.
$$
Using the integrating factor $(T-t)^{2c_{3/2}\varepsilon}$ (note that the exponent is $<1$) this becomes
$$
\frac{\d}{\d t}\left(\|u\|_{\dot H^{3/2}}^2(T-t)^{2c_{3/2}\varepsilon}\right)\le-\varepsilon^2(T-t)^{-(2-2c_{3/2}\varepsilon)}.
$$
Now drop the right-hand side and integrate from $\tau$ to $t$ to conclude that
\begin{align*}
\|u(t)\|_{\dot H^{3/2}}^2&\le\|u(\tau)\|_{\dot H^{3/2}}(T-\tau)^{2c_{3/2}\varepsilon}(T-t)^{2c_{3/2}\varepsilon}\\
&=C_\tau(T-t)^{2c_{3/2}\varepsilon},
\end{align*}
which contradicts (\ref{H32lower}) provided that $2c_{3/2}\varepsilon<1$, which we assumed above. It follows that there exist $t_k\to T$ such that
$$
\|u(t_k)\|_{\dot H^{5/2}}\ge (4c_{3/2})^{-1}t_k^{-1}
$$
and (\ref{limsup}) follows.
\end{proof}

Note that this bound does not use directly any differential inequality governing the evolution of $\|u\|_{\dot H^{5/2}}$.

\section{Strong blowup estimate in $\dot B^{5/2}_{2,1}$.}\label{sec:BB}

Although we have been unable to prove a strong lower bound in $\dot H^{5/2}$ at the correct rate (i.e.\ $\|u(t)\|_{\dot H^{5/2}}\ge c/(T-t)$) we can obtain such a bound in the Besov space $\dot B^{5/2}_{2,1}$, which has the same scaling. Again the proof relies on estimates of the nonlinear term, which we delay until Section \ref{sec:B}.

\begin{theorem}
 Suppose that $u$ is a classical solution of the Navier--Stokes existence with maximal existence time $T$. Then
  \begin{equation}\label{Besov-blowup}
  \|u(t)\|_{\dot B^{5/2}_{2,1}}\ge \frac{c}{T-t}.
  \end{equation}
\end{theorem}

\begin{proof}
  We consider the equation for $\loc_ku$, which can be rewritten (by adding and subtracting the term involving the summation in $i$) as
  $$
  \frac{\d}{\d t}\loc_ku-\Delta\loc_ku+ \left[\loc_k((u\cdot\nabla)u) -\sum_i\dot{S}_{k-1}u_i\partial_{i}\loc_k u\right]+\sum_i \dot{S}_{k-1}u_i\partial_{i}\loc_k u  =0,
  $$
  since $\loc_k$ and $\Delta$ commute. Taking the inner product in $L^2$ with $\loc_ku$ yields
  $$
  \frac{1}{2}\frac{\d}{\d t}\|\loc_ku\|_{L^2}^2+\|\nabla\loc_ku\|_{L^2}^2\le\left\|\loc_k((u\cdot\nabla)u) -\sum_i\dot{S}_{k-1}u_i\partial_{i}\loc_k u\right\|_{L^2}\|\loc_ku\|_{L^2}.
  $$
  We drop the second term on the left-hand side and divide by $\|\loc_ku\|_{L^2}$, to yield
  \begin{align*}
  \frac{\d}{\d t}\|\loc_ku\|_{L^2} &\le \left\|\loc_k((u\cdot\nabla)u) -\sum_i\dot{S}_{k-1}u_i\partial_{i}\loc_k u\right\|_{L^2}\\
  &\le d_k(t)2^{-5k/2}\|u\|_{\dot B^{5/2}_{2,1}}^2,
  \end{align*}
using Proposition  \ref{Besovestimate}, and  where $\sum d_k(t)=1$ for each $t$.

  We now multiply by $2^{5k/2}$ and sum to obtain
  $$
  \frac{\d}{\d t}\|u\|_{\dot B^{5/2}_{2,1}}\le c\|u\|_{\dot B^{5/2}_{2,1}}^2,
  $$
  from which (\ref{Besov-blowup}) follows at once via Lemma \ref{easy}.
\end{proof}

\section{Bounds for the nonlinear term in Sobolev spaces}\label{sec:H}

In this section we will prove the bound on the nonlinear term that we used in the proof of Theorem \ref{thm:H32}, namely
$$
|(\Lambda^{3/2}B(u,u),\Lambda^{3/2}u)|\le c_{3/2}\|u\|_{\dot H^{3/2}}^2\|u\|_{\dot H^{5/2}}.
$$
In fact we prove a somewhat more general result in Corollary \ref{wwan}, which in turn is a consequence of the following commutator estimate (cf.\ Kato \& Ponce, 1988; Fefferman et al., 2014).

%
%


\begin{proposition}
\label{prop:NSInequalityNorm}
Take $s \geq 1$ and $s_{1}, s_{2} > 0$ such that
\begin{equation}
1 \leq s_{1} < \tfrac{n}{2} + 1 \qquad \text{and} \qquad s_{1} + s_{2} = s + \tfrac{n}{2} + 1. \tag{\ref{eqn:Conditions}}
\end{equation}
Then there exists a constant $c$ such that for all $u, \BB \in \dot{H}^{s_{1}}(\R^{n}) \cap \dot{H}^{s_{2}}(\R^{n})$,
\[
\norm{\Lambda^{s}[(u \cdot \nabla)\BB] - (u \cdot \nabla) (\Lambda^{s} \BB)}_{L^{2}} \leq c ( \norm{u}_{\dot{H}^{s_{1}}} \norm{\BB}_{\dot{H}^{s_{2}}} + \norm{u}_{\dot{H}^{s_{2}}} \norm{\BB}_{\dot{H}^{s_{1}}} ).
\]
\end{proposition}


To prove Proposition~\ref{prop:NSInequalityNorm} we need two simple lemmas. A proof of the first can be found in Fefferman et al.\ (2014); the second is an immediate consequence of Berstein's inequality (see McCormick et al., 2013, for example).

\begin{lemma}
\label{lem:BasicIneq}
If $s \geq 1$ and $|b| < |a|/2$, then 
\[
\abs{|a|^{s} - |a-b|^{s}} \leq c |a-b|^{s-1} |b|,
\]
where $c=s3^{s-1}$.
\end{lemma}

\begin{lemma}
\label{lem:Bernstein}
There exists a constant $c$ such that, for any $k \in \Z$ and any $p,q$ with $1 \leq p \leq q \leq \infty$, if $\loc_{k} u \in L^{p}(\R^{n})$ then $\loc_{k} u \in L^{q}(\R^{n})$ and
\[
\norm{\loc_{k} u}_{L^{q}} \leq c 2^{kn(1/p-1/q)} \norm{\loc_{k} u}_{L^{p}}.
\]
\end{lemma}

We can now give the proof of Proposition \ref{prop:NSInequalityNorm}.

%

\begin{proof}[Proof of Proposition~\ref{prop:NSInequalityNorm}]
Write $u = \sum_{i \in \Z} \loc_{i} u$ and $\BB = \sum_{j \in \Z} \loc_{j} \BB$; then
\begin{align*}
\ff &= \Lambda^{s}[(u \cdot \nabla)\BB] - (u \cdot \nabla) (\Lambda^{s} \BB) \\
&= \sum_{j \in \Z}  \Lambda^{s} \left[ \left( \sum_{i \in \Z} \loc_{i} u \right) \nabla \loc_{j} \BB \right] - \left( \sum_{i \in \Z} \loc_{i} u \right) \nabla \Lambda^{s} \loc_{j}\BB \\
&= \sum_{j \in \Z}  \Lambda^{s} \left[ \left( \sum_{i = -\infty}^{j-10} \loc_{i} u \right) \nabla \loc_{j} \BB \right] - \left( \sum_{i = -\infty}^{j-10} \loc_{i} u \right) \nabla \Lambda^{s} \loc_{j}\BB \\
&\qquad + \sum_{j \in \Z}  \Lambda^{s} \left[ \left( \sum_{i = j-9}^{j+9} \loc_{i} u \right) \nabla \loc_{j} \BB \right] - \left( \sum_{i = j-9}^{j+9} \loc_{i} u \right) \nabla \Lambda^{s} \loc_{j}\BB \\
&\qquad + \sum_{i \in \Z}  \Lambda^{s} \left[ \loc_{i} u \left( \sum_{j=-\infty}^{i-10} \nabla \loc_{j} \BB \right) \right] - \loc_{i} u \left( \sum_{j=-\infty}^{i-10} \nabla \Lambda^{s} \loc_{j}\BB \right) \\
&=: \sum_{j \in \Z} \ff_{1,j} + \sum_{j \in \Z} \ff_{2,j} + \sum_{i \in \Z} \ff_{3,i}.
\end{align*}

Taking the Fourier transform of $\ff_{1,j}$, we have
\[
\hat{\ff}_{1,j} (\xi) = \int_{\R^{n}} \left( |\xi|^{s} - |\eta|^{s} \right) \sum_{i = -\infty}^{j-10} \widehat{\loc_{i} u} (\xi - \eta) \eta \widehat{\loc_{j} \BB} (\eta) \, \rd \eta.
\]
Since $i \leq j-10$, $|\xi - \eta| < |\eta|/2$, so by Lemma~\ref{lem:BasicIneq} we have
\[
|\hat{\ff}_{1,j} (\xi)| \leq \int_{\R^{n}} |\xi - \eta| \left| \sum_{i = -\infty}^{j-10} \widehat{\loc_{i} u} (\xi - \eta) \right| |\eta|^{s} \widehat{\loc_{j} \BB} (\eta) \, \rd \eta.
\]
Let $q_{1}, q_{2}$ satisfy $\frac{1}{q_{1}} + \frac{1}{q_{2}} = \frac{1}{2}$ and $2 < q_{1} < \frac{n}{s_{1} - 1}$, and let $p_{1}, p_{2}$ satisfy $\frac{1}{p_{i}} = \frac{1}{q_{i}} + \frac{1}{2}$. Noting that $1 + \frac{1}{2} = \frac{1}{p_{1}} + \frac{1}{p_{2}}$, by Young's inequality for convolutions we have
\[
\norm{\hat{\ff}_{1,j}}_{L^{2}} \leq \bignorm{|\zeta| \left| \sum_{i = -\infty}^{j-10} \widehat{\loc_{i} u} (\zeta) \right|}_{L^{p_{1}}} \bignorm{|\eta|^{s} \widehat{\loc_{j} \BB} (\eta)}_{L^{p_{2}}}.
\]
As $1 - s_{1} + n/q_{1} > 0$, by H\"{o}lder's inequality we have
\begin{align*}
\bignorm{|\zeta| \left| \sum_{i = -\infty}^{j-10} \widehat{\loc_{i} u} (\zeta) \right|}_{L^{p_{1}}} &\leq \bignorm{|\zeta|^{1 - s_{1}} \dInd_{\{ |\zeta| \leq 2^{j-10} \}}}_{L^{q_{1}}} \bignorm{|\zeta|^{s_{1}} \left| \sum_{i = -\infty}^{j-10} \widehat{\loc_{i} u} (\zeta) \right|}_{L^{2}} \\
&\leq c 2^{j(1 - s_{1} + n/q_{1})} \bignorm{u}_{\dot{H}^{s_{1}}}.
\end{align*}
For the other term, by H\"{o}lder's inequality,
\begin{align*}
\bignorm{|\eta|^{s} \widehat{\loc_{j} \BB} (\eta)}_{L^{p_{2}}} &\leq \bignorm{|\eta|^{s} \dInd_{\{ 2^{j-1} \leq |\zeta| \leq 2^{j+1} \}} }_{L^{q_{2}}} \bignorm{\widehat{\loc_{j} \BB} (\eta)}_{L^{2}} \\
&\leq c 2^{j(s + n/q_{2})} \bignorm{\loc_{j} \BB}_{L^{2}},
\end{align*}
hence
\begin{align*}
\norm{\ff_{1,j}}_{L^{2}} &\leq c \bignorm{u}_{\dot{H}^{s_{1}}} 2^{j(s - s_{1} + n/q_{1} + n/q_{2} + 1)} \bignorm{\loc_{j} \BB}_{L^{2}} \\
&\leq c \bignorm{u}_{\dot{H}^{s_{1}}} 2^{js_{2}} \bignorm{\loc_{j} \BB}_{L^{2}} \\
\end{align*}
and thus
\begin{equation}
\label{eqn:Part1}
\sum_{j \in \Z} \norm{\ff_{1,j}}_{L^{2}}^{2} \leq c \bignorm{u}_{\dot{H}^{s_{1}}}^{2} \bignorm{\BB}_{\dot{H}^{s_{2}}}^{2}.
\end{equation}

For the second term, since $\left( \sum_{i = j-9}^{j+9} \loc_{i} u \right) \nabla \loc_{j} \BB$ is localised in Fourier space in an annulus centred at radius $2^{j}$, we obtain
\begin{align*}
\norm{\ff_{2,j}}_{L^{2}} &\leq \bignorm{\Lambda^{s} \left[ \left( \sum_{i = j-9}^{j+9} \loc_{i} u \right) \nabla \loc_{j} \BB \right]}_{L^{2}} + \bignorm{\left( \sum_{i = j-9}^{j+9} \loc_{i} u \right) \nabla \Lambda^{s} \loc_{j}\BB}_{L^{2}} \\
&\leq c 2^{js} \sum_{i = j-9}^{j+9} \norm{\loc_{i} u}_{L^{4}} \norm{\nabla \loc_{j} \BB}_{L^{4}} + \sum_{i = j-9}^{j+9} \norm{\loc_{i} u}_{L^{4}} \norm{\nabla \Lambda^{s} \loc_{j} \BB}_{L^{4}} \\
&\leq c 2^{j(s + n/4)} \norm{\nabla \loc_{j} \BB}_{L^{2}} \sum_{i = j-9}^{j+9} 2^{in/4} \norm{\loc_{i} u}_{L^{2}} \\
&\leq c 2^{j(s + n/2 - s_{1})} \norm{\nabla \loc_{j} \BB}_{L^{2}} \sum_{i = j-9}^{j+9} 2^{j(s_{1} - n/4)} 2^{in/4} \norm{\loc_{i} u}_{L^{2}}
\end{align*}
using Bernstein's inequality (Lemma~\ref{lem:Bernstein}). Since $|i-j| \leq 9$, $2^{j(s_{1} - n/4)} \leq c 2^{i(s_{1} - n/4)}$, so
\[
\norm{\ff_{2,j}}_{L^{2}} \leq c 2^{j(s_{2} - 1)} \norm{\nabla \loc_{j} \BB}_{L^{2}} \sum_{i = j-9}^{j+9} 2^{is_{1}} \norm{\loc_{i} u}_{L^{2}},
\]
and thus
\begin{equation}
\label{eqn:Part2}
\sum_{j \in \Z} \norm{\ff_{2,j}}_{L^{2}}^{2} \leq c \bignorm{u}_{\dot{H}^{s_{1}}}^{2} \bignorm{\BB}_{\dot{H}^{s_{2}}}^{2}.
\end{equation}

For the third term, we use the Sobolev embedding
\[
\norm{\nabla u}_{L^{p}} \leq c \norm{u}_{\dot{H}^{s_{1}}}
\]
provided $p = \frac{2n}{n-2s_{1}+2}$. Using H\"{o}lder's inequality, we obtain
\begin{align*}
\norm{\ff_{3,i}}_{L^{2}} &\leq \bignorm{\Lambda^{s} \left[ \loc_{i} u \left( \sum_{j=-\infty}^{i-10} \nabla \loc_{j} \BB \right) \right]}_{L^{2}} + \bignorm{\loc_{i} u \left( \sum_{j=-\infty}^{i-10} \nabla \Lambda^{s} \loc_{j}\BB \right) }_{L^{2}} \\
&\leq 2^{is} \norm{\loc_{i} u}_{L^{n/(s_{1} - 1)}} \bignorm{\sum_{j=-\infty}^{i-10} \nabla \loc_{j} \BB}_{L^{2n/(n-2s_{1}+2)}} \\
&\qquad + \norm{\loc_{i} u}_{L^{n/(s_{1} - 1)}} \bignorm{\sum_{j=-\infty}^{i-10} \nabla \Lambda^{s} \loc_{j} \BB}_{L^{2n/(n-2s_{1}+2)}} \\
&\leq c 2^{i(s + n/2 + 1 - s_{1})} \norm{\loc_{i} u}_{L^{2}} \norm{\BB}_{\dot{H}^{s_{1}}} \\
&\leq c 2^{is_{2}} \norm{\loc_{i} u}_{L^{2}} \norm{\BB}_{\dot{H}^{s_{1}}}
\end{align*}
using Bernstein's inequality (Lemma~\ref{lem:Bernstein}) and the fact that $2^{js} \leq 2^{is}$. Hence
\begin{equation}
\label{eqn:Part3}
\sum_{i \in \Z} \norm{\ff_{3,i}}_{L^{2}}^{2} \leq c \bignorm{u}_{\dot{H}^{s_{2}}}^{2} \bignorm{\BB}_{\dot{H}^{s_{1}}}^{2}.
\end{equation}
Combining \eqref{eqn:Part1}, \eqref{eqn:Part2} and \eqref{eqn:Part3} yields the desired result.
\end{proof}

In particular, taking $s = s_{1} = n/2$ and $s_{2} = n/2 + 1$ in Proposition~\ref{prop:NSInequalityNorm} yields
\begin{align*}
&\norm{\Lambda^{n/2}[(u \cdot \nabla)\BB] - (u \cdot \nabla) (\Lambda^{n/2} \BB)}_{L^{2}} \\
&\qquad \qquad \qquad \leq c ( \norm{\nabla u}_{\dot{H}^{n/2}} \norm{\BB}_{\dot{H}^{n/2}} + \norm{u}_{\dot{H}^{n/2}} \norm{\nabla \BB}_{\dot{H}^{n/2}} ).
\end{align*}
The counterexample in the appendix to Fefferman et al.\ (2014) shows that one cannot remove the second term on the right-hand side, at least in the case $n=2$.

We will use this estimate in the form of the following corollary, which provides a partial generalisation of Lemma~1.1 from Chemin (1992).

\begin{corollary}\label{wwan}
 Take $s\ge 1$ and $s_{1}, s_{2} > 0$ such that
  $$
  1 \leq s_{1} < \tfrac{n}{2} + 1 \qquad \text{and} \qquad s_{1} + s_{2} = s + \tfrac{n}{2} + 1.
  $$
  Then there exists a constant $c$ such that for all $u, v \in \dot{H}^{s_{1}}(\R^{n}) \cap \dot{H}^{s_{2}}(\R^{n})$ with $\Div u = 0$,
\[
\abs{\inner{\Lambda^{s}[(u \cdot \nabla)v]}{\Lambda^{s}v}} \leq c ( \norm{u}_{\dot{H}^{s_{1}}} \norm{v}_{\dot{H}^{s_{2}}} + \norm{u}_{\dot{H}^{s_{2}}} \norm{v}_{\dot{H}^{s_{1}}} ) \norm{v}_{\dot{H}^{s}}.
\]
\end{corollary}

\begin{proof}
  Observe that since
  $$
  \inner{(u \cdot \nabla)\Lambda^{s}v}{\Lambda^{s}v}=0
  $$
  it follows that
  $$
  \inner{\Lambda^{s}[(u \cdot \nabla)v]}{\Lambda^{s}v}=\inner{\Lambda^s[(u\cdot\nabla)v]-(u\cdot\nabla)\Lambda^sv}{\Lambda^sv}
  $$
  and the inequality is an immediate consequence of Proposition \ref{prop:NSInequalityNorm}.
\end{proof}

Note that in particular for any $s\ge1$, if $\nabla\cdot u=0$ then
\begin{equation}\label{whatwewant}
\abs{\inner{\Lambda^s[(u \cdot \nabla)u]}{\Lambda^su}} \leq c \|u\|_{\dot{H}^s}\|u\|_{\dot H^{s+1}} \|u\|_{\dot{H}^{3/2}}
\end{equation}
whenever the right-hand side is finite.

\section{Bounds for the nonlinear term in Besov spaces}\label{sec:B}



Much like the Sobolev embeddings, Besov spaces enjoy certain embeddings with the correct exponents. We quote the two embeddings we will use most frequently.

\begin{proposition}[Proposition~2.20 in Bahouri et al (2011)]
\label{prop:BesovEmbedding}
Let $1 \leq p_{1} \leq p_{2} \leq \infty$ and $1 \leq r_{1} \leq r_{2} \leq \infty$. For any real number $s$, we have the continuous embedding
\[
\dot{B}^{s}_{p_{1}, r_{1}}(\R^{n}) \hookrightarrow \dot{B}^{s - n (1/p_{1} - 1/p_{2})}_{p_{2}, r_{2}}(\R^{n}).
\]
\end{proposition}

\begin{proposition}[Proposition~2.39 in Bahouri et al (2011)]
\label{prop:BesovLebesgueEmbedding}
For $1 \leq p \leq q \leq \infty$, we have the continuous embedding
\[
\dot{B}^{n/p-n/q}_{p,1}(\R^{n}) \hookrightarrow L^{q}(\R^{n}).
\]
\end{proposition}


\subsection{Homogeneous Paradifferential Calculus}

Let $u$ and $v$ be tempered distributions in $\Sch'_{h}(\R^{n})$. We have
\[
u = \sum_{j' \in \Z} \loc_{j'} u \quad \text{and} \quad v = \sum_{j \in \Z} \loc_{j} v,
\]
so, at least formally,
\[
uv = \sum_{j, j' \in \Z} \loc_{j'} u \loc_{j} v.
\]
One of the key techniques of paradifferential calculus is to break the above sum into three parts, as follows: define
\[
\dot{T}_{u} v := \sum_{j \in \Z} \dot{S}_{j-1} u \loc_{j} v,
\]
and
\[
\dot{R}(u,v) := \sum_{|k-j| \leq 1} \loc_{k} u \loc_{j} v.
\]
At least formally, the following \emph{Bony decomposition} holds true:
\[
uv = \dot{T}_{u} v + \dot{T}_{v} u + \dot{R}(u,v).
\]
We now state two standard estimates on $\dot{T}$ and $\dot{R}$ that we will use in proving our a priori estimates.

\begin{lemma}[Theorem 2.47 from Bahouri et al (2011)]
\label{lem:BonyTEstimate}
Let $s \in \R$ and $t < 0$. There exists a constant $C = C(s,t)$ such that for any $p, r_{1}, r_{2} \in [1,\infty]$, $u \in \dot{B}^{t}_{p,r_{1}}$ and $v \in \dot{B}^{s}_{p,r_{2}}$,
\[
\norm{\dot{T}_{u} v}_{\dot{B}^{s+t}_{p,r}} \leq C \norm{u}_{\dot{B}^{t}_{\infty,r_{1}}} \norm{v}_{\dot{B}^{s}_{p,r_{2}}}
\]
with $\frac{1}{r} = \min \left\{ 1, \frac{1}{r_{1}} + \frac{1}{r_{2}} \right\}$.
\end{lemma}

\begin{lemma}[Theorem 2.52 from Bahouri et al (2011)]
\label{lem:BonyREstimate}
Let $s_{1}, s_{2} \in \R$ such that $s_{1} + s_{2} > 0$. There exists a constant $C = C(s_{1}, s_{2})$ such that, for any $p_{1}, p_{2}, r_{1}, r_{2} \in [1,\infty]$, $u \in \dot{B}^{s_{1}}_{p_{1},r_{1}}$ and $v \in \dot{B}^{s_{2}}_{p_{2},r_{2}}$,
\[
\norm{\dot{R}(u,v)}_{\dot{B}^{s_{1}+s_{2}}_{p,r}} \leq C \norm{u}_{\dot{B}^{s_{1}}_{p_{1},r_{1}}} \norm{v}_{\dot{B}^{s_{2}}_{p_{2},r_{2}}}
\]
provided that
\[
\frac{1}{p} := \frac{1}{p_{1}} + \frac{1}{p_{2}} \leq 1 \quad \text{and} \quad \frac{1}{r} := \frac{1}{r_{1}} + \frac{1}{r_{2}} \leq 1.
\]
\end{lemma}


We also require the following Lemma, which is a particular case of Lemma~2.100 from Bahouri et al (2011).

\begin{lemma}
\label{lem:BesovCommutator}
Let $-1 - n/2 < \sigma < 1 + n/2$ and $1 \leq r \leq \infty$. Let $\vv$ be a divergence-free vector field on $\R^{n}$, and set $Q_{j} := [ (\vv \cdot \Grad), \loc_{j} ] f$. There exists a constant $C = C(\sigma, n)$, such that
\[
\bignorm{\left( 2^{j\sigma} \norm{Q_{j}}_{L^{2}} \right)_{j}}_{\ell^{r}} \leq C \norm{\Grad \vv}_{\dot{B}^{n/2}_{2, \infty} \cap L^{\infty}} \norm{f}_{\dot{B}^{\sigma}_{2,r}}.
\]
\end{lemma}



%
%


\subsection{Main Estimate in Besov spaces}

We are now ready for the main estimate in Besov spaces.

\begin{proposition}\label{Besovestimate}
  If $u\in \dot B^{n/2+1}_{2,1}$ then
\begin{equation}\label{Besovone}
\left\|\loc_k((u\cdot\nabla)u) - \sum_i\dot{S}_{k-1}u_i\partial_{i}\loc_k u\right\|_{L^2}\lesssim d_k 2^{-k(n/2+1)}\|u\|_{\dot{B}^{n/2+1}_{2,1}}^2,
\end{equation}
with $\sum_k d_k =1$.
\end{proposition}

\begin{proof}

%

Notice that the $l$-th coordinate of $(u\cdot\nabla)u$ is given by $\sum_i u_i \partial_{i} u_l$, and so we have
$$
(u\cdot \nabla u)_l = \sum_i \dot{T}_{u_i}\partial_{i} u_l+ \sum_i \dot{T}_{\partial_{i} u_l}u_i +\sum_i\dot{R}(u_i,\partial_{i} u_l)
$$
Recall that by definition
$$
\dot{T}_{u_i}\partial_{i} u_l=\sum_j \dot{S}_{j-1} u_i \loc_j \partial_{i}u_l,
$$
and so we can rewrite $\loc_k\dot{T}_u\nabla u_l$ as follows
\begin{align}
  \sum_i\loc_k\dot{T}_{u_i}\partial_{i} u_l&=\sum_i\dot{S}_{k-1}u_i\partial_{i}\loc_k u_l\\
  &+\sum_i\sum_j(\dot{S}_{j-1}u_i-\dot{S}_{k-1}u_i)\partial_{i}\loc_k\loc_ju_l\\
  &+\sum_i\sum_j[\loc_k,\dot{S}_{j-1}u_i\partial_{i}]\loc_ju_l.
\end{align}

And so we obtain the following expression for the $l$-th component of the term we want to estimate
\begin{align}
\Big(\loc_k((u\cdot\nabla)u) &- \sum_i\dot{S}_{k-1}u_i\partial_{i}\loc_k u \Big)_l= \nonumber \\
&=\sum_i\sum_j(\dot{S}_{j-1}u_i-\dot{S}_{k-1}u_i)\partial_{i}\loc_k\loc_ju_l \label{t1}\\
  &+\sum_i\sum_j[\loc_k,\dot{S}_{j-1}u_i\partial_{i}]\loc_ju_l \label{t2} \\
	& +\sum_i \loc_k \dot{T}_{\partial_{i} u_l}u_i \label{t3}\\
	&+ \sum_i \loc_k\dot{R}(u_i,\partial_{i} u_l) \label{t4}
\end{align}

We will show that $L^2$ norm of each of the four terms in the right hand side is controlled by a constant multiple of $d_k 2^{-k(n/2+1)}\|u\|_{\dot{B}^{n/2+1}_{2,1}}^2$, hence obtaining the result.

For $(\ref{t1}$), ignoring the summation in $i$ for now we have
$$
\sum_j(\dot{S}_{j-1}u_i-\dot{S}_{k-1}u_i)\partial_{i}\loc_k\loc_ju_l=\loc_{k-1}u_i\loc_k\loc_{k+1}\partial_{i} u_l- \loc_{k-2}u_i\loc_k\loc_{k-1}\partial_{i} u_l,
$$
and so (now summing in $i$ as well)
\begin{align*}
\|\mbox{expression (\ref{t1})}\|_{L^2}&\lesssim 2^k\|\loc_{k-1}u\|_{L^\infty}\|\loc_ku_l\|_{L^2}\\
&\qquad+2^k\|\loc_{k-2}u\|_{L^\infty}\|\loc_ku_l\|_{L^2}\\
& \lesssim \|\loc_k u_l\|_{L^2} \|u\|_{\dot{B}^{n/2+1}_{2,1}} \\
& \lesssim d_k 2^{-k(n/2+1)} \|u\|_{\dot{B}^{d/2+1}_{2,1}}\|u_l\|_{\dot{B}^{n/2+1}_{2,1}}
\end{align*}
 since
$$
2^k\|\loc_ku\|_{L^\infty}\le\|u\|_{\dot B^1_{\infty,\infty}}\lesssim\|u\|_{\dot{B}^{n/2+1}_{2,1}}.
$$
Above we have used the definition of $\dot B^1_{\infty,\infty}$ and the corresponding embedding from Proposition \ref{prop:BesovLebesgueEmbedding}, and also \eqref{eqn:StandardTrick} to find
$$
\|\loc_ku\|_{L^2}\lesssim d_k 2^{-k(n/2+1)}\|u\|_{\dot{B}^{n/2+1}_{2,1}}.
$$

\noindent To treat \eqref{t2}, define $Q_k=\sum_j[\loc_k,\dot S_{j-1}u_i\partial_{i}]\loc_j u_l$, then applying Lemma \ref{lem:BesovCommutator} we have
$$
 \Big{\|}2^{k(n/2+1)}\|Q_k\|_{L^2}\Big{\|}_{\ell^1} \lesssim \|\nabla u \|_{\dot{B}^{n/2}_{2,\infty}\cap L^{\infty}} \|u\|_{\dot{B}^{n/2+1}_{2,1}} \lesssim \|u\|_{\dot{B}^{n/2+1}_{2,1}}^2,
$$
since $\dot B^{n/2}_{2,1}$ embeds continuously into $L^{\infty}$ and $\dot B^{n/2}_{2,\infty}$ (see Proposition \ref{prop:BesovEmbedding} and \ref
{prop:BesovLebesgueEmbedding}).
Hence
$$
\|Q_k\|_{L^2} \lesssim d_k 2^{-k(n/2+1)}\|u\|_{\dot{B}^{n/2+1}_{2,1}}^2.
$$

To estimate \eqref{t3} we use Lemma \ref{lem:BonyTEstimate}  and the embeddings from Proposition \eqref{prop:BesovEmbedding}; we have
\begin{align*}
\|\dot T_{\partial_{i}u_l} u_i\|_{\dot B^{n/2+1}_{2,1}}&\lesssim \|\nabla u_l\|_{\dot B^{0}_{\infty,\infty}}\|u_i\|_{\dot B^{n/2+1}_{2,1}}\\
&\lesssim\|u\|^2_{\dot B^{n/2+1}_{2,1}}.
\end{align*}
Using \eqref{eqn:StandardTrick} we find
$$
\|\loc_k\dot T_{\partial_{i}u_l}u_i\|_L^2 \leq d_k 2^{-k(n/2+1)} \|u\|^2_{\dot B^{n/2+1}_{2,1}}.
$$

Finally we consider \eqref{t4}; using Lemma \ref{lem:BonyREstimate} with $p=2$, $(p_1,p_2)=(\infty,2)$, $s_1=1$, $r=1$, $(r_1,r_2)=(\infty,1)$, $s_2=n/2$, we obtain
\begin{align*}
  \|\dot R(u_i, \partial_{i} u_l)\|_{\dot B^{n/2+1}_{2,1}}&\lesssim\|u_i\|_{\dot B^1_{\infty,\infty}}\|\nabla u_l\|_{\dot B^{n/2}_{2,1}}\\
  &\lesssim \|u\|_{\dot B^{n/2+1}_{2,1}}^2,
\end{align*}
since by Proposition \ref{prop:BesovEmbedding}
$$
\dot B^s_{p_1,r_1}\subset \dot B^{s-d(1/p_1-1/p_2)}_{p_2,r_2}.
$$
Again, by \eqref{eqn:StandardTrick}  we find

$$
\|\loc_k \dot R(u_i, \partial_{i} u_l)\|_{L^2} \leq d_k 2^{-k(n/2+1)} \|u\|^2_{\dot B^{n/2+1}_{2,1}}
$$

Combining these estimates yields (\ref{Besovone}).\end{proof}

\section{Conclusion}

Lower bounds in $\dot H^{3/2}$ are now available from a number of sources. Whether it is possible to obtain a strong lower bound in $\dot H^{5/2}$ remains an interesting open question, as does the possibility of obtaining bounds at the optimal rate in $\dot H^s$ for $s>5/2$.

\section*{Acknowledgments}

This work arose as a result of a visit to Warwick by YZ in October 2013. DSM was a member of Warwick's `MASDOC' doctoral training centre, funded by EPSRC grant EP/HO23364.1; JCR was supported by an EPSRC Leadership Fellowship EP/G007470/1, which also funded EJO's sabbatical year at Warwick during which most of this work was completed, and a postdoctoral position for AVL; JLR is partially supported by the European Research Council, grant no. 616797.

\section*{References}

\parindent0pt\parskip5pt

H. Bahouri, J.-Y. Chemin, and R. Danchin. {\it Fourier analysis and nonlinear partial differential equations}. Springer-Verlag (2011).

J. Benameur. On the blow-up criterion of 3D Navier-Stokes equations. {\it J. Math. Anal. Appl.} {\bf 371} (2010) 719–727.

J. Benameur. On the blow-up criterion of the periodic incompressible fluids. {\it Math. Methods Appl. Sci.} {\bf 36} (2013) 143–153.

A. Cheskidov and M. Zaya. Lower Bounds of Potential Blow-Up Solutions of the Three-dimensional Navier-Stokes Equations in $\dot{H}^\frac{3}{2}$. arXiv:1503.01784.

J.C. Cortissoz, J.A. Montero, C.E. Pinilla. On lower bounds for possible blow-up solutions to the periodic Navier-Stokes equation. {\it J. Math. Phys.} {\bf  55}  (2014) 033101.

C.L. Fefferman, D.S. McCormick, J.C. Robinson, and J.L. Rodrigo. Higher order commutator estimates and local existence for the non-resistive {MHD} equations and related models. {\it J. Funct. Anal.} {\bf 267} (2014) 1035--1056.

Y. Giga. Solutions for semilinear parabolic equations in Lp and regularity of weak solutions of the Navier-Stokes system. {\it J. Differ. Equ.} {\bf 62} (1986) 186-–212.

T. Kato and G. Ponce. Commutator estimates and the Euler and Navier--Stokes equations. {\it Comm. Pure Appl. Math.} {\bf 41} (1988) 891--907.

J. Leray. Sur le mouvement d'un liquide visqueux emplissant l'espace. {\it Acta Math.} {\bf 63} (1934) 193-–248 .

D.S. McCormick, J.C. Robinson, and J.L. Rodrigo. Generalised Gagliardo--Nirenberg inequalities using weak Lebesgue spaces and BMO. {\it Milan J. Math.} {\bf 81} (2013) 265--289.

J.A. Montero. Lower bounds for possible blow-up solutions for the Navier--Stokes equations revisited. arXiv: 1503.03063.

J.C. Robinson, W. Sadowski, and R.P. Silva. Lower bounds on blow up solutions of the three-dimensional Navier-Stokes equations in homogeneous Sobolev spaces. {\it J. Math. Phys.} {\bf 53} (2012) 115618

%
%

\end{document}